\documentclass[11pt,a4]{article}
\usepackage{lipsum}
\usepackage{amsfonts}
\usepackage{graphicx}
\usepackage{epstopdf}
\usepackage{algorithmic}

\usepackage{bm}
\usepackage{bbm}
\usepackage{amsmath}
\usepackage{amssymb}
\usepackage{amsthm}

\newcommand{\D}{\mathrm{d}}
\newcommand{\I}{\mathrm{i}}
\newcommand{\E}{\mathrm{e}}
\newcommand{\eps}{\varepsilon}

\newcommand{\T}{\intercal}

\newcommand*\diff{\mathop{}\!\D}

\newcommand{\abs}[1]{\lvert#1\rvert} 
\newcommand{\Abs}[1]{\left\lvert#1\right\rvert} 
\newcommand{\norm}[1]{\lVert#1\rVert}

\DeclareMathAlphabet{\mathsfit}{\encodingdefault}{\sfdefault}{m}{sl}
\SetMathAlphabet{\mathsfit}{bold}{\encodingdefault}{\sfdefault}{bx}{n}

\newcommand{\sone}{\mathbbm{1}}

\newcommand{\sN}{\mathbb{N}}

\newcommand{\sR}{\mathbb{R}}

\newcommand{\sZ}{\mathbb{Z}}

\newcommand{\fS}{\mathcal{S}}

\newtheorem{thm}{Theorem}

\newtheorem{prop}{Proposition}
\newtheorem{rmk}{Remark}
\newtheorem{assump}{Assumption}

\title{Bunching instability and asymptotic properties in epitaxial growth with elasticity effects: continuum model
}

\author{
    Tao Luo
    \footnotemark[1]\and
    Yang Xiang
    \footnotemark[2]\and
    Nung Kwan Yip
    \footnotemark[3]
}

\footnotetext[1]{School of Mathematical Sciences, Institute of Natural Sciences, MOE-LSC, and Qing Yuan Research Institute, Shanghai Jiao Tong University, Shanghai, 200240, P.R. China. (luotao41@sjtu.edu.cn)}
\footnotetext[2]{Department of Mathematics, Hong Kong University of Science and Technology, Clear Water Bay, Hong Kong.
(maxiang@ust.hk)}
\footnotetext[3]{Department of Mathematics, Purdue University, West Lafayette, IN 47907, USA.
(yip@math.purdue.edu)}

\date{Received: date / Accepted: date}

\begin{document}
\maketitle
\allowdisplaybreaks
\noindent
\begin{abstract}
    We study the continuum epitaxial model for elastic interacting atomic steps on vicinal surfaces proposed by Xiang and E \cite{Xiang2002Derivation,Xiang2004Misfit}. The non-local term and the singularity complicate the analysis of its PDE. In this paper, we first generalize this model to the Lennard--Jones $(m,n)$ interaction between steps. Based on several important formulations of the non-local energy, we prove the existence, symmetry, unimodality, and regularity of the energy minimizer in the periodic setting. In particular, the symmetry and unimodality of the minimizer implies that it has a bunching profile. Furthermore, we derive the minimum energy scaling law for the original continnum model. All results are consistent with the corresponding results proved for discrete models by Luo et al. \cite{Luo2016Energy,Luo2021Energy}.
\end{abstract}

\newpage
\section{Introduction}

In epitaxial film growth, various surface morphological instabilities may occur due to elastic effects caused by surface defects such as steps as well as the mismatch between the lattice constants of the substrate and the film. Among these instabilities, the step bunching instability attracts materials scientists' interest for its potential to facilitate the fabrication of nanostructures on the vicinal surfaces. Below the roughening transition temperature, a vicinal surface consists of a succession of terraces and atomic height steps, while the angle between this surface and the crystallographic plane is as small as a few tenths of one degree. Due to elastic interactions, the steps are attracted to each other and coalesce into step bunches whose size increase gradually in the evolution \cite{Tersoff1995Step, Duport1995Growth}. 

Based on the Burton, Cabrera, and Frank (BCF) theory \cite{Burton1951growth} which assumes the diffusion of adatoms on the terraces until their attachment to the steps, 
Tersoff et al. \cite{Tersoff1995Step,Liu1998Self} proposed a discrete model step dynamics on the vicinal surfaces, including two elastic effects of steps: the attractive one as a force monopole $\sum_{j\neq i}\frac{1}{x_j-x_i}$ and the repulsive one as a force dipole $\sum_{j\neq i}\frac{1}{(x_j-x_i)^3}$. More precisely, their model reads as
\begin{align}
    \frac{\D x_i}{\D t}
    &= a^2F_{\mathrm{ad}}\frac{x_{i+1}-x_{i-1}}{2}+
    \frac{a^2\rho_0D }{k_BT}\left(\frac{\mu_{i+1}-\mu_i}{x_{i+1}-x_i}-\frac{\mu_i-\mu_{i-1}}{x_i-x_{i-1}}\right),\quad i\in\sZ,\label{eq..Tersoff02}\\
    \mu_i
    &= \sum_{j\neq i}\left[\frac{\alpha_1}{x_j-x_i}-\frac{\alpha_2}{(x_j-x_i)^3}\right],\quad i\in\sZ.\label{eq..f_i.02}
\end{align}
where $x_i$ is the position of the $i$th step, $a$ is the lattice constant, $F_\mathrm{ad}$ is adatom flux, $\alpha_1$ ($\alpha_2$) is the strength of monopole (dipole)  interaction, and $\rho_0$, $D$, $k_B$, and $T$ are the equilibrium adatom density on a step without elastic interactions, the diffusion constant on the terrace, Boltzmann constant and temperature, respectively. 
Duport et al. \cite{Duport1995New,Duport1995Growth} also studied a related model including the elastic interaction between the adatoms and steps and the Schweobel barriers. Alhough Tersoff's discrete model performs well numerically, it is mathematically interesting to see how such a model leads to step-bunching in a highly non-linear regime. This was first proved by Luo et al. in \cite{Luo2016Energy} and then generalized to step-like systems \cite{Luo2021Energy} with step interacting under Lennard--Jones type potential. In particular, the authors formulated the problem into a discrete variational problem and showed that the minimizer has a bunching structure, i.e., steps tend to concentrate into a narrow band. Several scaling laws for the minimum energy as well as the minimal distance between adjacent steps was also provided. 

By taking asymptotic expansion and keeping the first two orders of Tersoff's discrete model \cite{Tersoff1995Step,Duport1995Growth}, Xiang and E \cite{Xiang2002Derivation, Xiang2004Misfit} derived a continuum model for epitaxial growth with elasticity which reads as
\begin{equation}
    h_t=\frac{a^2\rho_\mathrm{ad} D \pi\alpha_1}{k_B T}\partial_{xx}\left\{
    -H(h_x)-\eta\left(\frac{1}{h_x}+\gamma h_x\right)h_{xx}
    \right\},\quad x\in \sR,
\end{equation}
where $h$ is the height of the vicinal surface, two physical parameters are $\eta=\frac{a}{2\pi}$ and $\gamma=\frac{\pi^2 l_\mathrm{eq}^2}{a^2}=\frac{\alpha_2 \pi^2}{\alpha_1 a^2}$, the equilibrium distance is $l_\mathrm{eq}=\sqrt{\frac{\alpha_2}{\alpha_1}}$, and the Hilbert transform $H$ is defined as follows
\begin{equation}
  H(f)(x)=\frac{1}{\pi}\mathrm{P.V.}\int_{-\infty}^{\infty}\frac{f(y)}{x-y}\diff{y}.
\end{equation}
For the coherence of our presentation, we take no flux assumption that $F_\mathrm{ad}=0$ and the convention that the vicinal surface is monotonically increasing: $h_x\geq 0$. Note that this monotonicity assumption is opposite to \cite{Tersoff1995Step} and \cite{Xiang2002Derivation}, where they all assumed the decreasing surface profile.
Later, Xu and Xiang further extended this to 2+1 dimensions \cite{Xu2009Derivation} (See also \cite{Zhu2009Continuum}). These continuum models capture the step-bunching phenomenon in terms of the linear instability and numerical simulation. 

For the analysis of the 1+1 dimensional continnum model by Xiang \cite{Xiang2002Derivation}, the existence and regularity of weak solution was studied by Dal Maso et al. \cite{DalMaso2014Analytical}, Fonseca et al. \cite{Fonseca2015Regularity} and very recently improved by Gao et al. \cite{Gao2020regularity}. 
For the connection between the discrete and the continuum models, assuming all the contributin tersm are of the same length scale, Gao et al. \cite{Gao2017continuum} proved the first order convergence rate of a modified discrete model to the strong solution of the limiting PDE. 
However, due to the non-local integral, sigularity, and high order derivatives in the model, the mathematical understanding of this continuum model is far from complete. In particular, the characterizations of the bunching structure, such as the unimodal profile and its energy scaling law, are not known yet. 

Let us also mention some related analysis works on continuum epitaxial growth model without step bunching phenomenon. Al Hajj Shehadeh et al. \cite{AlHajjShehadeh2011evolution} considered both discrete and continuum models in the attachment-detachment-limited regime including only the force dipole interaction between nearest steps, and hence there is no step bunching phenomenon. Li and Liu \cite{Li2003Thin} investigated epitaxial growth with or without slope selection systematically and obtained the well-posedness, perturbation analysis, as well as numerical simulation. 

In this paper, we are going to study the appearance of step bunching from an energetic point of view. We first follow Xiang's derivation and obtain a continuum model under Lennard--Jones $(m,n)$ interaction. We show rigorously that the chemical potential of this continuum is the leading order of its discrete counterpart. Next, for this generalized model, we establish the existence, symmetry, unimodality, and regularity of its energy minimizer in the periodic setting. The unimodality and symmetry of the minimizer implies that it has a bunching-like profile. Then we further prove for the mininmizer a non-trivial energy scaling law. All results are consistent with to those proved in \cite{Luo2021Energy} for general Lennard--Jones interaction on a discrete system.


\section{Main results}
We study a generalized continuum model for epitaxial growth on the vicinal surfaces with steps interacting with the Lennard--Jones $(m,n)$ potential
\begin{equation}
  V(x):=\left\{
  \begin{array}{ll}
      -\frac{\alpha_1}{m}|x|^{-m}+\frac{\alpha_2}{n}|x|^{-n}, & -1<m<1<n,\, m\neq 0, \\
      \alpha_1\log|x|+\frac{\alpha_2}{n}|x|^{-n}, & m=0,1<n,
  \end{array}
  \right.\label{eq..LennardJones}
\end{equation}
for $-1<m<1<n$. Its derivative for $x>0$ reads as
$V'(x)=\alpha_1 x^{-m-1}-\alpha_2 x^{-n-1}$.
Then the equilibrium distance between two successive steps $l_\mathrm{e}$, satisfying $V'(l_\mathrm{e})=0$, i.e.,
$l_\mathrm{e}=(\alpha_2/\alpha_1)^{\frac{1}{n-m}}$.
The physically reasonable scale for $l_\mathrm{e}$ is of order $a$, the lattice constant. For $-1<m<1<n$, the generalized epitaxial growth dynamics of the steps $\{x_i\}_{i\in\sZ}$ (with the convention $\cdots<x_i<x_{i+1}<\cdots$) reads as
\begin{align}
    \frac{\D x_i}{\D t}
    &= a^2F_{\mathrm{ad}}\frac{x_{i+1}-x_{i-1}}{2}+
    \frac{a^2\rho_0D }{k_BT}\left(\frac{\mu_{i+1}-\mu_i}{x_{i+1}-x_i}-\frac{\mu_i-\mu_{i-1}}{x_i-x_{i-1}}\right),\quad i\in\sZ,\label{eq..Tersoff}\\
    \mu_i
    &= \alpha_1\sum_{k=1}^{\infty}\left[(x_{i+k}-x_i)^{-m-1}-(x_i-x_{i-k})^{-m-1}\right]\label{eq..f_i.mn}\\
    &~~~~+\alpha_2\sum_{k=1}^{\infty}\left[(x_{i+k}-x_i)^{-n-1}-(x_i-x_{i-k})^{-n-1}\right].\nonumber
\end{align}
The convention $\cdots<x_i<x_{i+1}<\cdots$ leads to the height function $h(x)$ being monotonically increasing in the continuum model.

\textbf{(i) Derivation of the generalized continuum model.}
We first derive the continuum model corresponding to the above generalized discrete model. At equilibrium, the discrete system reads as
\begin{equation*}
    \mu_i=0,\quad i\in\sZ.
\end{equation*}
From now on, we rescale it into the chemical potential per unit length, i.e.,
\begin{equation}
    \mu_i^{\mathrm{a}}
    :=\sigma_i^{(m)}-\frac{\alpha_2}{\alpha_1}\sigma_i^{(n)}=a\sum\limits_{j\in\sZ, j\neq i}\left(\frac{x_j-x_i}{\abs{x_j-x_i}^{m+2}}-\frac{\alpha_2}{\alpha_1}\frac{x_j-x_i}{\abs{x_j-x_i}^{n+2}}\right).
\end{equation}
This is the atomistic chemical potential, later we will show that it converges to the continuum chemical potential under some assumptions.
Here for $s>-1$, we define
\begin{align}
    \sigma_i^{(s)}
    := a\sum_{j\in\sZ,j\neq i}\frac{x_j-x_i}{\abs{x_j-x_i}^{s+2}}
    =a\sum_{k=1}^\infty G_s(ka;ia),\label{eq..sigma.definition}
\end{align}
where for any strictly monotonically increasing function $x(h)$, we define the function
\begin{equation}
    G_s(h;\xi):=(x(\xi+h)-x(\xi))^{-s-1}-(x(\xi)-x(\xi-h))^{-s-1}.\label{eq..G(h).definition}
\end{equation}
The Taylor theorem and dominated convergence theorem is sufficient to derive a continuum limit mimicking the behavior of $\sigma^{(n)}$ for $n>1$. However, it is more subtle to derive the continuum counterpart for $-1<m<1$ because of the singularity and the non-local effects. We will use the Euler--Maclaurin expansion for singular integrals. For a reference on this, see Appendix \ref{sec..EulerMaclaurin}.

We show that for any sufficiently smooth surface $h(x)$ (see Theoerm \ref{thm..consistency}) the atomistic chemical potential of this model converges to its continuum counterpart:
\begin{equation}
    \mu(x)=-\mathrm{P.V.}\int_{-\infty}^{+\infty}\frac{(x-y)h_x(y)}{\abs{x-y}^{m+2}}\diff{y}-\eps^{1-m}\left(\frac{1}{h_x^{1-m}}+\gamma h_x^{n-1}\right)h_{xx}.
\end{equation}
Consequently, the governing equation of the continuum model (with some physical constants dropped) reads as:
\begin{equation}
    h_t=\mu_{xx}
    =\left[-\mathrm{P.V.}\int_{-\infty}^{+\infty}\frac{(x-y)h_x(y)}{\abs{x-y}^{m+2}}\diff{y}-\eps^{1-m} \left(\frac{1}{h_x^{1-m}}+\gamma h_x^{n-1}\right)h_{xx}\right]_{xx}.\label{eq..general.PDE}
\end{equation}
Here and in the rest of this paper ``$\mathrm{P.V.}$'' means the integral is defined by the principle value. For some infinite series which are not absolutely convergent, its summation $\sum_{k\in\sZ}$ should also be understood in the sense of the principle value $\lim_{N\to\infty}\sum_{\abs{k}\leq N}$.
In the derivation of the continuum model, we assume that $\gamma=O(1)$ and $\eps\to0$ for $-1<m<1<n$, where
\begin{align}
    \gamma
    &:= \frac{\alpha_2 a^m}{\alpha_1 a^n}\frac{(n+1)\zeta(n)}{(m+1)\abs{\zeta(m)}},\\
    \eps^{1-m}
    &:= (m+1)\abs{\zeta(m)} a^{1-m}.
\end{align}
The zeta function $\zeta(s)$ is a function of a complex variable s that analytically continues the sum of the Dirichlet series $\zeta(s)=\sum_{k=1}^\infty k^{-s}$, which converges when $\mathrm{Re}(s)>1$.
We remark that $\zeta(m)<0$ for $-1<m<1$ and $\zeta(n)>0$ for $1<n$. Hence $\gamma$ and $\eps$ are positive. We will always suppose that $\gamma$ is independent of $a$ (or $\eps$).
Then we have the following result under a technical assumption (See Assumption \ref{assump..regularity.vicinal.surface}).
\begin{thm}[consistency of atomistic and continuum chemical potentials]\label{thm..consistency}
    Suppose that Assumption \ref{assump..regularity.vicinal.surface} holds. If $-1<m<1<n$, then
    \begin{equation}
        \lim_{\eps\to 0}\frac{\mu^\mathrm{a}-\mu}{\eps^{1-m}}=0.
    \end{equation}
\end{thm}

\textbf{(ii) Existence, symmetry, and regularity of the step-bunching profile.}
Next, we study the continuum model in a periodic setting. Without loss of generality, we denote the periodic domain as $\Omega:=[-\frac{1}{2},\frac{1}{2}]$. Suppose that the average step density (or slope) is a given constant $A>0$. For any height function $h(x)$ defined on $\sR$, we write $\tilde{h}(x)=h(x)-Ax$ for height deviation from the reference state and $\rho(x)=h_x(x)$ for step density.
For $k\in\sN$, we denote $W^{k,p}_\#(\Omega)$ the Sobolev space of functions whose weak derivatives of order less than $k$ are $1$-periodic and in the space $L^p(\Omega)$. In particular, we also write $H^{k}_\#(\Omega)=W^{k,2}_\#(\Omega)$.
Define the Hilbert space
\begin{equation}
    V:=\left\{\tilde{h}\in H^1_\#(\Omega)\mid \int_{\Omega}\tilde{h}(x)\diff{x}=0\right\}.
\end{equation}
Then we define the solution space
\begin{equation}
    X
    := \left\{h\in H^1_\mathrm{loc}(\sR)\mid\tilde{h}(x)=h(x)-A x\in V,\,\, h_x(x)\geq 0, \,\,\text{a.e.}\,\, x\in\sR\right\}.
\end{equation}
It is easy to see that $X$ is a convex closed subset of $V+\{Ax\}$.
Given a function $h\in X$ and $-1<m<1<n$, we define its energy funcitonal for the continuum model 
\begin{equation}
    E[h]:=-I_m[h]+\eps^{1-m}\int_{\Omega} \Phi_{m,n}(h_x)\diff{x},\label{eq..continuum.energy.definition}
\end{equation}
where the non-local energy and local energy density read as
\begin{align}
    I_m[h]
    &:= \frac{1}{2}\int_{\Omega}h(x)\mathrm{P.V.}\int_{\sR}\frac{(x-y)\tilde{h}_x(y)}{\abs{x-y}^{m+2}}\diff{y}\diff{x},\label{eq..I_m[h].definition}\\
    \Phi_{m,n}(\xi)
    &:= \left\{
    \begin{array}{ll}
        \frac{1}{m(m+1)}\xi^{m+1}+\frac{\gamma}{n(n+1)}\xi^{n+1}, & -1<m<1<n, m\neq 0, \xi>0\\
        \xi\log \xi+\frac{\gamma}{n(n+1)}\xi^{n+1}, & m=0, 1<n, \xi>0,\\
        0, & -1<m<1<n, \xi=0,\\
        +\infty, & -1<m<1<n, \xi<0,
    \end{array}
    \right.\label{eq..Phi_m,n.definition}
\end{align}
respectively. 
If $\xi>0$, then 
\begin{equation}\label{eq..derivativePhi}
    \Phi_{m,n}'(\xi)
    = \left\{
    \begin{array}{ll}
        \frac{1}{m}\xi^{m}+\frac{\gamma}{n}\xi^{n}, & -1<m<1<n, m\neq 0,\\
        \log \xi+1+\frac{\gamma}{n}\xi^{n}, & m=0, 1<n.
    \end{array}
    \right.
\end{equation}
By direct calculation, one can see that
\begin{equation}
	\frac{\delta E}{\delta h}(x)
	=-\mathrm{P.V.}\int_{-\infty}^{+\infty}\frac{(x-y)\tilde{h}_x(y)}{\abs{x-y}^{m+2}}\diff{y}-\eps^{1-m}\left(\frac{1}{h_x^{1-m}(x)}+\gamma h_x^{n-1}(x)\right)h_{xx}(x)
	=\mu(x),
\end{equation} i.e, the continuum chemical potential is the first variation of the total energy.
\begin{rmk}
    In particular, if $(m,n)=(0,2)$, then $\mu(x)=-\pi H(h_x)-\eps(\frac{1}{h_x}+\gamma h_x)h_{xx}$. The negative sign in $O(\eps)$ term is due to our convention $h_x\geq 0$ which is opposite to the monotonically decreasing assumption used in the previous works, e.g., \cite{Tersoff1995Step, Xiang2002Derivation}. (See also Remark \ref{rmk..fractional.Laplacian}.)
\end{rmk}

We show that the total energy can be represented in a kernel-based convolution, which is essential in the estimate of the step density $\rho=h_x$ (see Proposition \ref{prop..kernel-basedRepresentation}):
\begin{equation}
    E[h]=-\frac{1}{2}\int_{\Omega}\int_{\Omega} \rho(x)\rho(y)K_m(x-y)\diff{x}\diff{y}+\eps^{1-m}\int_{\Omega} \Phi_{m,n}(\rho(x))\diff{x},
\end{equation}
where the kernel $K_m(\cdot)$ is $L^1(\Omega)$, non-negative, $1$-periodic, even, and strictly decreasing on $(0,\frac{1}{2})$.

Let us briefly introduce the notion of rearrangement. 
For any measurable set $S\subset\sR$ with measure $\abs{S}$, let $S^*=[-\tfrac{1}{2}|S|,\tfrac{1}{2}|S|]$ be its symmetric rearrangement. Given $h\in X$, we define the symmetric decreasing rearrangement of $\rho=h_x\in L^2$ as
\begin{equation}
    \rho^*=\int_0^\infty \sone_{\{y\mid \rho(y)>t\}^*}\diff{t},
\end{equation}
where $\sone_S$ denotes the indicator function of the set $S$. 
Geometrically, the measure of every level set of $\rho^*$ equals the one of $\rho$. 

We prove the existence of the minimizer by the direct method in the calculus of variations. Then the kernel-based representation with a rearrangement theorem leads to the symmetry and unimodality of the energy minimizer in one periodic. For $-1<m\leq 0$, we further obtain the regularity of the energy minimizer.
To sum up, we have the following results. 

\begin{thm}[bunching profile: existence]\label{thm..existence}
    If $A>0$, $\eps>0$, and $-1<m<1<n$, then there exists a global minimizer $h$ of energy \eqref{eq..continuum.energy.definition} in $X$ , i.e.,
    \begin{equation}
        E[h]=\min_{g\in X}E[g].
    \end{equation}
\end{thm}

\begin{thm}[bunching profile: symmetry and unimodality]\label{thm..sysmmetry}
	If $A>0$, $\eps>0$, and $-1<m<1<n$, then the global energy minimizer $h\in X$ of $E$ satifies, up to a translation,
    \begin{equation}
       \rho=\rho^*.
    \end{equation}
    In other words, the step density $\rho$ of the minimizer is symmetric and unimodal. As a result, the support of $\rho$ in the period $\Omega$ is a closed interval centered at zero, i.e., $\mathrm{supp}(\rho):=\{x\in\Omega\mid\overline{\{\rho(x)\neq0\}}=[-R_0,R_0]$ for some $R_0\in(0,\frac{1}{2}]$. 
\end{thm}

\begin{thm}[bunching profile: regularity]\label{thm..regularity}
    If $A>0$, $\eps>0$, and $-1<m\leq 0$, $1<n$, then for the global energy minimizer $h\in X$ of $E$ with $\mathrm{supp}(h_x)=[-R_0,R_0]$, we have $h$ is smooth on $[-R_0, R_0]$ and $h$ is constant on $(R_0,1-R_0)$. 
\end{thm}

These theorems together tells us that the system energetically prefers the symmetric and unimodal profile which is also smooth. Such a surface is the step-bunching profile in the literature. We provide Figure \ref{figure..bunching} to illustrate such a bunching profile.

\begin{figure}
  \begin{center}
  \mbox{(a)\includegraphics[height=0.33\textwidth]{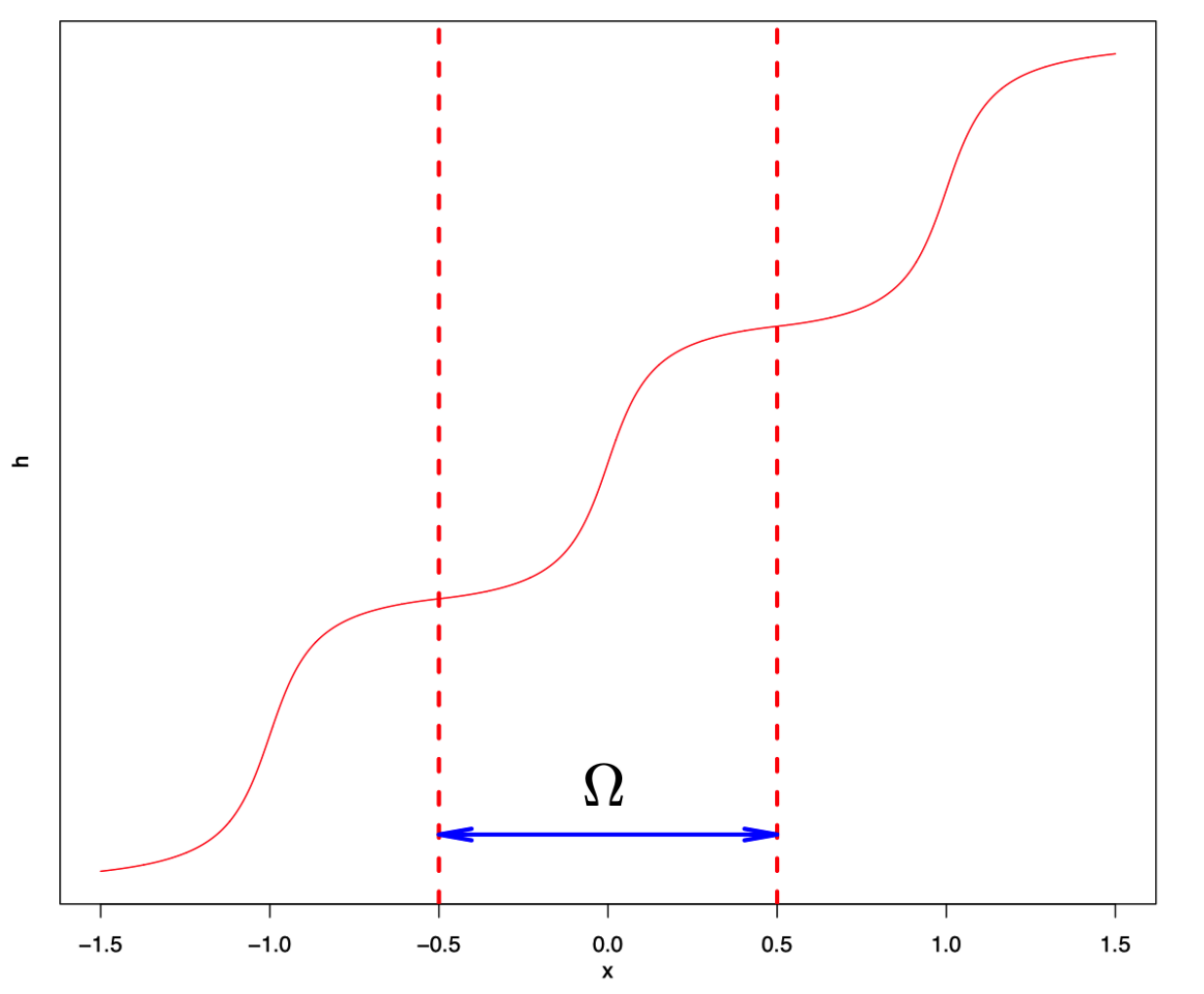}
\hspace{10pt}
  (b)\includegraphics[height=0.33\textwidth]{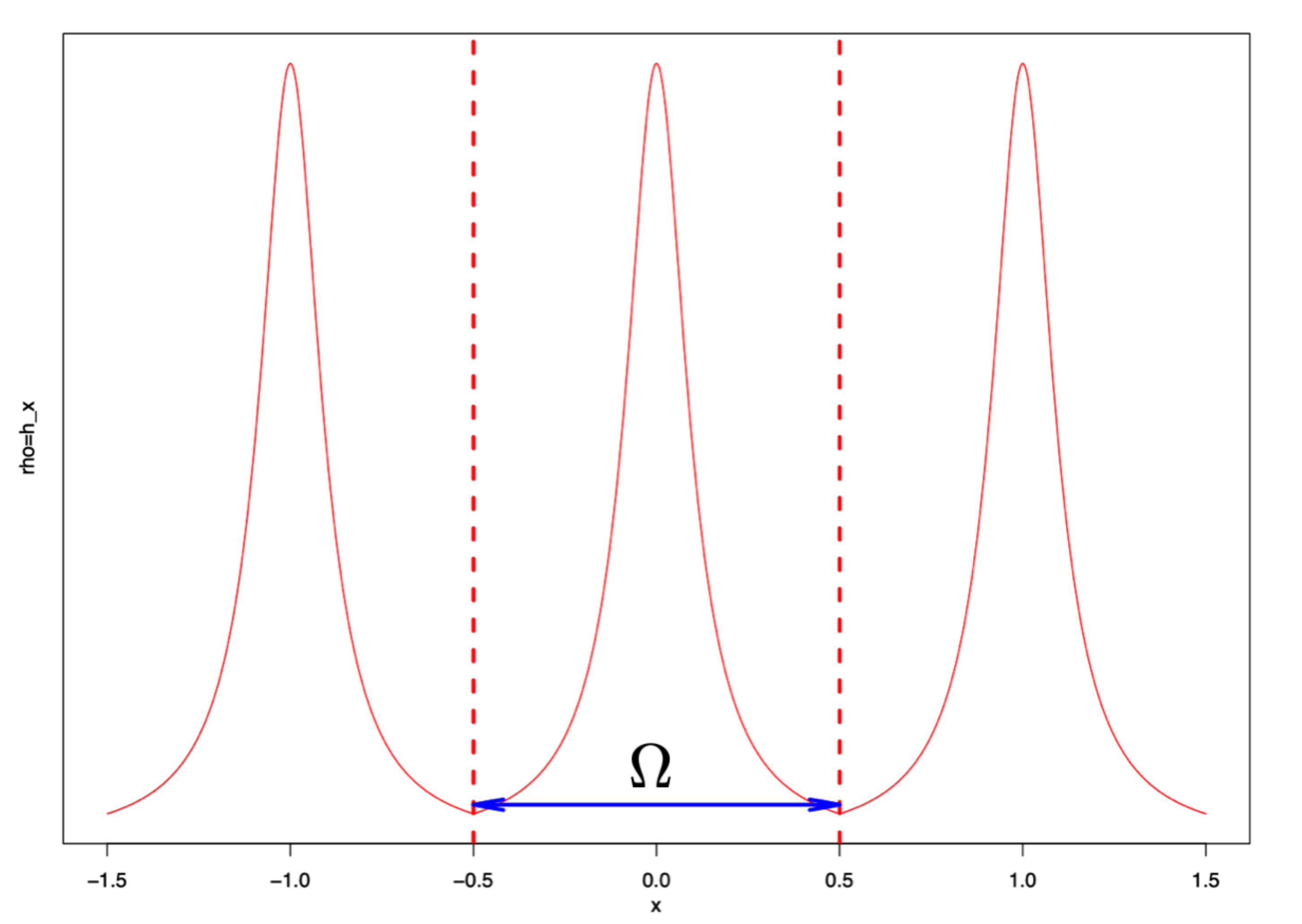}}
  \end{center}
  \caption{(a) A continuous vicinal surface $h$ with step-bunching in the periodic setting. Theorem \ref{thm..sysmmetry} shows there is a single bunch in each period if the profile attains the minimum energy. 
  (b) The step density $\rho=h_x$ (derivative of the height function $h$) of the bunching profile (a).}
  \label{figure..bunching}
\end{figure}


\textbf{(iii) Minimum energy scaling law of the step-bunching profile.}
For any function $h\in X$, we have another representation of the total energy, up to a null Lagrangian $W[h]$,
\begin{align}
    E[h]
    &= -\tilde{I}_m[h]-W[h]+\eps^{1-m}\int_{\Omega} \Phi_{m,n}(h_x)\diff{x},\label{eq..continuum.energy.definition2}\\
    \tilde{I}_m[h]
    &:= \frac{1}{2}\int_{\Omega}\tilde{h}(x)\mathrm{P.V.}\int_{\sR}\frac{(x-y)\tilde{h}_x(y)}{\abs{x-y}^{m+2}}\diff{y}\diff{x}.\label{eq..I_m[h].definition2}
\end{align}
Then we derive a Fourier series-based representation of $\tilde{I}_m[h]$ (See Propostion \ref{prop..EquivalentFormulationNonlocalEnergy}). Thanks to the regularity, this representation holds for the energy minimizer. Consequently, we show the following energy scaling law for the original continuum model \cite{Xiang2002Derivation}. 

\begin{thm}[bunching profile: minimum energy scaling law]\label{thm..energy.scaling.law}
    If $A>0$, $m=0$, and $1<n$, then there exists an $\eps_0$ and positive constants $C,C'$ such that for any $\eps<\eps_0$ the following energy scaling law holds
    \begin{equation}
            -\frac{A^2}{n}\abs{\log \eps}-C\leq\inf_{h\in X}E[h]\leq  -\frac{A^2}{n}\abs{\log \eps}+C'.
    \end{equation}
\end{thm}

Although we use totally different proofs in this paper, all results are consistent with those of the discrete models \cite{Luo2016Energy,Luo2021Energy}.

The rest of this paper is organized as follows. In section 3, we derive a generalized continuum model corresponding to the Lennard--Jones $(m,n)$ interaction by analyzing the aymptotics of chemical potential (Theorem \ref{thm..consistency}). In section 4, we study the minimization problem and establish the existence, symmetry, unimodality, and regularity of the bunching profile of the generalized continuum model (Theorems \ref{thm..existence}, \ref{thm..sysmmetry}, and \ref{thm..regularity}). In section 5, the minimum energy scaling law is obtained for the original continuum model (Theorem \ref{thm..energy.scaling.law}).

\section{Generalized continuum model}\label{sec..generalized.cont.model}
In this section, we derive the generalized continuum model by asymptotic analysis of the discrete model and show their consistency in terms of chemical potentials. Our derivation essentially follows \cite{Xiang2002Derivation}, but the result is rigorous under the following assumption:
\begin{assump}[regularity of the vicinal surface]\label{assump..regularity.vicinal.surface}
    Suppose that the step location as a function of the vicinal surface height $x(h)$ satisfies the following conditions:
    \begin{enumerate}
        \item[(1)] $x(h)$ is real analytic. Moreover, for any $\xi\in \sR$, the Taylor expansion holds for all $h\in\sR$: $x(\xi+h)=\sum_{j=0}^\infty \frac{1}{j!}x^{(j)}(\xi) h^j$.
        \item[(2)] $p_0:=\inf_{h\in\sR}x_h(h)>0$.
    \end{enumerate}
\end{assump}

The Taylor expansion in the first condition will be helpful to calculate the Euler--Maclaurin expansion when we apply Theorem \ref{thm..SigularEulerMaclaurin}. The second condition corresponds the fact that the step density cannot be infinity.
Since our aim here is to present a general strategy to derive a continuum model from a discrete one, we are not going to pursuit the least regularity assumption. The validation and properties of the new continuum model on a generic vicinal surface, which may or may not satisfies Assumption \ref{assump..regularity.vicinal.surface}, will be left to the following sections and future works.

The following proposition provides our asymptotic expansion techniques, which is an extension of the derivation in \cite{Xiang2002Derivation}. Under Assumption \ref{assump..regularity.vicinal.surface}, the asymptotic expansion is rigorously proved. Moreover, Proposition \ref{prop..asymptotic.expansion.chemical.potential} provides a unified procedure to asymptotically expanding the chemical potential for $-1<m<1<n$ in all physically meaningful regimes.
\begin{prop}[asymptotic expansions for chemical potential]\label{prop..asymptotic.expansion.chemical.potential}
    Suppose that Assumption \ref{assump..regularity.vicinal.surface} holds. If $-1<m<1<n$, then at $\xi=ia$, we have
    \begin{align}
        \sigma_{i}^{(n)}
        &=-(n+1)\zeta(n)a^{1-n}x_{hh}(\xi)x_{h}^{-n-2}(\xi)+o(a^{1-n}),\\
        \sigma_i^{(m)}
        &=\int_0^\infty G_m(h;\xi)\diff{h} -(m+1)\zeta(m)a^{1-m}x_{hh}(\xi)x_h^{-m-2}(\xi)+O(a^2).\label{eq..-1<s<1.force.consistency}
    \end{align}
\end{prop}
We remark that the improper integral of $G_s$ is defined as
\begin{equation}
    \int_0^\infty G_s(h;\xi)\diff{h}
    =\lim_{M\to+\infty,\delta\to0^+}\left\{
    \int_{\xi+\delta}^{\xi+M}\frac{\diff{h}}{(x(h)-x(\xi))^{s+1}}-
    \int_{\xi-M}^{\xi-\delta}\frac{\diff{h}}{(x(\xi)-x(h))^{s+1}}
    \right\}.
\end{equation}
\begin{proof}
    (1)
    For $n>1$, we have
    \begin{align*}
        \lim_{a\to 0} a^{n-1} \sigma_{i}^{(n)}
        &= \lim_{a\to 0} a^n\sum_{k=1}^{\infty}\left[\frac{1}{(x_{i+k}-x_i)^{n+1}}-\frac{1}{(x_{i}-x_{i-k})^{n+1}}\right]\\
        &= -\lim_{a\to 0} a^n\sum_{k=1}^{\infty}\left[\frac{(x_{i+k}-x_i)^{n+1}-(x_{i}-x_{i-k})^{n+1}}{(x_{i+k}-x_i)^{n+1}(x_{i}-x_{i-k})^{n+1}}\right]\\
        &= -\lim_{a\to 0} a^n\sum_{k=1}^\infty\frac{
        (ka)^{n+2}(n+1)(x_h(\xi_k))^n x_{hh}(\xi'_k)
        }{(ka)^{2n+2}\left(\frac{x_{i+k}-x_i}{ka}\right)^{n+1}\left(\frac{x_{i}-x_{i-k}}{ka}\right)^{n+1}}\\
        &= -a^n\sum_{k=1}^\infty k^{-n}\lim_{a\to 0}\frac{
        (n+1)(x_h(\xi_k))^n x_{hh}(\xi'_k)
        }{\left(\frac{x_{i+k}-x_i}{ka}\right)^{n+1}\left(\frac{x_{i}-x_{i-k}}{ka}\right)^{n+1}}\\
        &= -(n+1)\zeta(n) x_{hh}(\xi) x_{h}^{-n-2}(\xi).
    \end{align*}
    Here $\xi_k,\xi'_k\in((i-k)a,(i+k)a)$. In the third and the last equality, we have used the mean value theorem. In fourth equality, we apply the dominated convergence theorem and exchange the limit with the sum because all derivatives of $x$ are continuous and $\left(\frac{x_{i+k}-x_i}{ka}\right)^{n+1}\left(\frac{x_{i}-x_{i-k}}{ka}\right)^{n+1}\geq (\inf_{h\in\sR}x_h(h))^{2n+2}=p_0^{2n+2}>0$.

    (2) If $-1<m<1$, we define
    \begin{equation*}
        g(h)=\left\{
        \begin{array}{ll}
            h^m G_m(h;\xi), & h>0,\\
            -(m+1)x_{hh}(\xi)x_{h}^{-m-2}(\xi), & h=0.
        \end{array}
        \right.
    \end{equation*}
    It is easy to check that $g(h)$ is infinitely differentiable on $[0,\infty)$. Moreover, we will show that the Maclaurin expansion of $g(h)$ holds in a small neighbor of $0$. We use the Taylor expansion and obtain $\frac{x(\xi+h)-x(\xi)}{h}=\sum_{r=0}^{\infty}c_{r+1}h^r$ and $\frac{x(\xi)-x(\xi-h)}{h}=\sum_{r=0}^{\infty}c_{r+1}(-h)^r$, where $c_r=\frac{1}{r!}x^{(r)}(\xi)$.
    Thanks to the positivity of $x_h$ and the mean value theorem, we have $\frac{x(\xi+h)-x(\xi)}{h}\geq p_0>0$ and $\frac{x(\xi)-x(\xi-h)}{h}\geq p_0>0$ for all $h>0$. Note that $f(\xi)=\xi^{-(m+1)}$ is analytic for $\xi\geq p_0$ with $m+1>0$. Since the composition of two analytic functions remains analytic, we may write
    \begin{equation*}
        \left(\frac{x(\xi+h)-x(\xi)}{h}\right)^{-m-1}
        =:\sum_{r=0}^\infty c'_r h^r,\quad h\in\sR.
    \end{equation*}
    Hence we have for $h>0$
    \begin{align*}
        g(h)
        &= h^{-1}\left\{\left(\frac{x(\xi+h)-x(\xi)}{h}\right)^{-m-1}-\left(\frac{x(\xi)-x(\xi-h)}{h}\right)^{-m-1}\right\}\\
        &= h^{-1}\sum_{r=0}^\infty c'_r [h^r-(-h)^r]= \sum_{r=0}^{\infty}2c'_{2r+1} h^{2r}.
    \end{align*}
    Taking limit $h\to 0$, the above equality still holds. By the uniqueness of the Maclaurin expansion, we conclude that $2c'_{2r+1}=\frac{1}{(2r)!}g^{(2r)}(0)$, $g^{(2r+1)}(0)=0$ for $r\geq 0$ and that
    \begin{equation}
        g(h)=\sum_{r=0}^\infty \frac{g^{(2r)}(0)}{(2r)!}h^{2r},\,\,h\geq 0.\label{eq..g(h).Maclaurin}
    \end{equation}

    Applying Theorem \ref{thm..SigularEulerMaclaurin} with $G_m(h;\xi)$ and $g(h)$, we have the asymptotic expansion
    \begin{align*}
        \int_{0}^{\infty}G_m(h;\xi)\diff{h}
        &= a \sum_{j=1}^{\infty}G_m(h_j;\xi)-\zeta(m)g(0)a^{1-m}-\zeta(m-1)g'(0)a^{2-m}+R_2,\\
        R_{2}
        &= a^{2}\int_{0}^{\infty}\frac{\bar{B}_{2}(h/a)-B_{2}}{2}\frac{\D^{2}}{\D h^{2}}
        \left\{\int_{0}^{h}g^{(2)}(h')h^{-m}(h-h')\diff{h}'
        \right\}\diff{h}.
    \end{align*}
    Note that $\sigma_i^{(m)}=a \sum_{j=1}^{\infty}G_m(h_j;\xi)$, $-\zeta(m)g(0)a^{1-m}=(m+1)\zeta(m)x_{hh}(\xi)x_{h}^{-m-2}(\xi)a^{1-m}$ and $\zeta(m-1)g'(0)a^{2-m}=0$. To prove the proposition, it is sufficient to show that $\abs{R_2}\leq Ca^2$.
    Since $\bar{B}_2(\cdot)$ are bounded on $\sR$ and is independent of $g$, it follows that
    \begin{equation*}
        \abs{R_2}\leq C a^2\int_0^\infty\Abs{\frac{\D^2}{\D h^2}\left\{
        \int_{0}^{h}g^{(2)}(h')h^{-m}(h-h')\diff{h}'
        \right\}}\diff{h}.
    \end{equation*} 
    By direct calculation, we have $g^{(2)}(h')=\sum_{r=1}^\infty\frac{g^{(2r)}(0)}{(2r-2)!}(h')^{2r-2}$ and
    $\int_{0}^{h} (h')^{2r-2}(h-h')\diff{h}'=\frac{1}{2r(2r-1)}h^{2r}$ for $r\geq 1$.
    Therefore
    \begin{align*}
        \frac{\D^2}{\D h^2}\left\{
        \int_{0}^{h}g^{(2)}(h')h^{-m}(h-h')\diff{h}'
        \right\}
        &= \frac{\D^2}{\D h^2}\left\{h^{-m}\sum_{r=1}^\infty\frac{g^{(2r)}(0)}{(2r-2)!}
        \int_{0}^{h} (h')^{2r-2}(h-h')\diff{h}'
        \right\}\\
        &= \frac{\D^2}{\D h^2}\left\{h^{-m}\sum_{r=1}^\infty\frac{g^{(2r)}(0)}{(2r)!}h^{2r}
        \right\}.
    \end{align*}
    Now we split the estimate of the integral in $\abs{R_2}$ into two parts, the first part is
    \begin{align*}
        &~~~~\int_{0}^1\Abs{\frac{\D^2}{\D h^2}\left\{h^{-m}\sum_{r=1}^\infty\frac{g^{(2r)}(0)}{(2r)!}h^{2r}
        \right\}}\diff{h}\\
        &\leq \int_{0}^1\Abs{\sum_{r=1}^\infty\frac{(2r-m)(2r-m-1)g^{(2r)}(0)}{(2r)!}h^{2r-m-2}}\diff{h}\\
        &\leq \sum_{r=1}^\infty\frac{\abs{g^{(2r)}(0)}}{(2r)!}(2r-m)\Abs{\int_{0}^1(2r-m-1)h^{2r-m-2}\diff{h}}\\
        &\leq \sum_{r=1}^\infty\frac{\abs{g^{(2r)}(0)}}{(2r)!}(2r+1)\\
        &\leq \sum_{r=1}^\infty\frac{\abs{g^{(2r)}(0)}}{(2r-1)!}+\sum_{r=1}^\infty\frac{\abs{g^{(2r)}(0)}}{(2r)!}
        \leq C,
    \end{align*}
    where in the last inequality, we use the facts that 
    the power series \eqref{eq..g(h).Maclaurin} and its derivative both converge absolutely at $h=1$.

    The integral in the second part of $\abs{R_2}$ is
    \begin{align*}
        \int_{1}^\infty\Abs{\frac{\D^2}{\D h^2}\left\{h^{-m}\sum_{r=1}^\infty\frac{g^{(2r)}(0)}{(2r)!}h^{2r}
        \right\}}\diff{h}
        &\leq \int_1^\infty\Abs{\frac{\D^2}{\D h^2}\left\{h^{-m}\left[g(h)-g(0)\right]\right\}}\diff{h}\leq C.
    \end{align*}
    where we use the fast decay property of the integrand at $+\infty$. As a result, $\abs{R_2}\leq Ca^2$. The proof is completed.
\end{proof}
\begin{rmk}
    If $m=0$, then $\int_0^\infty G_0(h;\xi)\diff{h}=\int_{-\infty}^\infty \frac{1}{x(\xi+h)-x(\xi)}\diff{h}$ and \eqref{eq..-1<s<1.force.consistency} reduces to the result in \cite{Xiang2002Derivation}:
    \begin{equation*}
        \sigma_i^{(0)}
        = \int_{-\infty}^\infty \frac{1}{x(\xi+h)-x(\xi)}\diff{h} +\frac{1}{2}a x_{hh}x_h^{-2}+O(a^{2}).
    \end{equation*}
\end{rmk}

\begin{rmk}
    Note that $x_h=h_x^{-1}$ and $x_{hh}=-h_{xx}h_x^{-3}$. Suppose that $\frac{\alpha_2}{\alpha_1}=O(a^{n-m})$ which is the physically interesting regime. For $1<m<n$, we have the leading order of the chemical potential satisfies the equation
    \begin{equation}
        (m+1)\zeta(m)h_{xx}h_x^{m-1}-\frac{\alpha_2a^m}{\alpha_1 a^{n}}(n+1)\zeta(n)h_{xx}h_x^{n-1}=0.
    \end{equation}
    Therefore the solution is linear, i.e., $h(x)=Ax+B$, hence there is no bunching. This is consistent with the case of the well-known Lennard--Jones $(6,12)$ potential \cite{Luo2021Energy}.
\end{rmk}

The Cauchy principal value is unaffected by the change of variable from $h$ to $x$.
\begin{prop}[change of variable]\label{prop..change.of.variable}
    Suppose that Assumption \ref{assump..regularity.vicinal.surface} holds. If $-1<m<1$, then we have the following equality at $x(\xi)=x_i$
    \begin{equation}
        \int_0^\infty G_m(h;\xi)\diff{h}
        =-\mathrm{P.V.}\int_{-\infty}^{+\infty}\frac{(x_i-x)h_x(x)}{\abs{x_i-x}^{m+2}}\diff{x}.\label{eq..change.of.variable}
    \end{equation}
\end{prop}
\begin{proof}
    By definition
    \begin{align*}
        &~~\int_0^\infty G_m(h;\xi)\diff{h}\\
        &= \lim_{M\to+\infty,\delta\to0^+}
        \left\{\int_{\xi+\delta}^{\xi+M}
        \frac{\D h}{(x(h)-x(\xi))^{m+1}}-\int_{\xi-M}^{\xi-\delta}\frac{\D h}{(x(\xi)-x(h))^{m+1}}\right\},\\
        &~~-\mathrm{P.V.}\int_{-\infty}^{+\infty}\frac{(x_i-x)h_x}{\abs{x_i-x}^{m+2}}\diff{x}\\
        &= \lim_{M'\to+\infty,\delta'\to0^+}
        \left\{\int_{x_i+\delta'}^{x_i+M'}\frac{h_x}{(x-x_i)^{m+1}}\diff{x}
        -\int_{x_i-M'}^{x_i-\delta'}\frac{h_x}{(x_i-x)^{m+1}}\diff{x}
        \right\}.
    \end{align*}
    For fixed $\delta$ and $M$, we change the independent variable from $h$ to $x$:
    \begin{align*}
        \int_{\xi+\delta}^{\xi+M}\frac{\D h}{(x(h)-x(\xi))^{m+1}}
        &= \int_{x(\xi+\delta)}^{x(\xi+M)}\frac{h_x(x)}{(x-x_i)^{m+1}}\diff{x}\\
        &= \left(\int_{x(\xi+\delta)}^{x_i+x_h(\xi)\delta}
        +\int_{x_i+x_h(\xi)\delta}^{x_i+\frac{M}{A}}
        +\int_{x_i+\frac{M}{A}}^{x(\xi+M)}
        \right)\frac{h_x(x)}{(x-x_i)^{m+1}}\diff{x}.
    \end{align*}
    Note that $\abs{h_x}=\abs{x_h}^{-1}\leq p_0^{-1}$.
    Thus
    \begin{align*}
        \int_{x(\xi+\delta)}^{x_i+x_h(\xi)\delta}\frac{h_x}{(x-x_i)^{m+1}}\diff{x}
        &= \int_{x_i+x_h(\xi)\delta+O(\delta^2)}^{x_i+x_h(\xi)\delta}\frac{h_x}{(x-x_i)^{m+1}}\diff{x}
        =O(\delta^{1-m}),\\
        \int_{x_i+\frac{M}{A}}^{x(\xi+M)}\frac{h_x}{(x-x_i)^{m+1}}\diff{x}
        &=\Abs{x(\xi+M)-x_i-\frac{M}{A}}O(M^{-m-1})\\
        &=\Abs{\frac{\xi+M}{A}+\tilde{x}(\xi+M)-\frac{\xi}{A}-\tilde{x}(\xi)-\frac{M}{A}}O(M^{-m-1})\\
        &=O(M^{-m-1}).
    \end{align*}
    Therefore, we have
    \begin{equation}
        \int_{\xi+\delta}^{\xi+M}
        \frac{\D h}{(x(h)-x(\xi))^{m+1}}
        =\int_{x_i+x_h(\xi)\delta}^{x_i+\frac{M}{A}}\frac{h_x}{(x-x_i)^{m+1}}\diff{x}+O(\delta^{1-s})+O(M^{-m-1}),\label{eq..change.of.variable.+}
    \end{equation}
    and similarly,
    \begin{equation}
        -\int_{\xi-M}^{\xi-\delta}
        \frac{\D h}{(x(\xi)-x(h))^{m+1}}
        =-\int_{x_i-\frac{M}{A}}^{x_i-x_h(\xi)\delta}\frac{h_x}{(x_i-x)^{m+1}}\diff{x}+O(\delta^{1-m})+O(M^{-m-1}).\label{eq..change.of.variable.-}
    \end{equation}
    To finish the proof, we combine \eqref{eq..change.of.variable.+} and \eqref{eq..change.of.variable.-}, then let $\delta\to0$ and $M\to+\infty$.
\end{proof}
The assumption $-1<m<1$ is required to guarantee that the truncation error in the Cauchy principle value is negligible.

\begin{proof}[Proof of Theorem \ref{thm..consistency}]
    Combining Propositions \ref{prop..asymptotic.expansion.chemical.potential} and \ref{prop..change.of.variable}, we have at $x=x_i$
    \begin{align*}
        \mu^\mathrm{a}
        &= -\mathrm{P.V.}\int_{-\infty}^{+\infty}\frac{(x-y)h_x(y)}{\abs{x-y}^{m+2}}\diff{y} -(m+1)\zeta(m)a^{1-m}x_{hh}x_h^{-m-2}+O(a^2)\\
        &~~~~+\frac{\alpha_2}{\alpha_1}\left[a^{1-n}(n+1)\zeta(n) x_{hh} x_{h}^{-n-2}+o(a^{1-n})\right].
    \end{align*}
    By using $x_h^{-1}=h_x$ and $x_{hh}=-h_{xx}h_x^{-3}$, we obtain
    \begin{align*}
        & -(m+1)\zeta(m)a^{1-m}x_{hh}x_h^{-m-2}=(m+1)\zeta(m)a^{1-m}\frac{h_{xx}}{h_{x}^{1-m}}=-\eps^{1-m}\frac{h_{xx}}{h_{x}^{1-m}},\\
        & \frac{\alpha_2}{\alpha_1}a^{1-n}(n+1)\zeta(n) x_{hh} x_{h}^{-n-2}
        =-\frac{\alpha_2a^m}{\alpha_1a^n} (n+1)\zeta(n)a^{1-m}h_{xx}h_x^{n-1}
        =-\gamma\eps^{1-m}h_{xx}h_x^{n-1}.
    \end{align*}
    Therefore, $\mu^\mathrm{a}=\mu+O(a^2)+o(\frac{\alpha_2}{\alpha_1}a^{1-n})=\mu+o(\eps^{1-m})$.
\end{proof}

Before we finish this section, we show the connection of our non-local integral to other definitions of the non-local integral of $h$.
\begin{prop}[equivalent definitions of non-local integral]\label{prop..relation.two.definition.non-local}
    Suppose that Assumption \ref{assump..regularity.vicinal.surface} holds. If $-1<m<1$, then
    \begin{equation}
        \mathrm{P.V.}\int_{-\infty}^{+\infty}\frac{(x-y)f_x(y)}{\abs{x-y}^{m+2}}\diff{y}
        =(m+1)\mathrm{P.V.}\int_{-\infty}^{+\infty}\frac{f(x)-f(y)}{\abs{x-y}^{m+2}}\diff{y}
        .\label{eq..relation.two.definition.non-local}
    \end{equation}
\end{prop}
\begin{proof}
    Integration by parts leads to
    \begin{align*}
        & 
        \int_{x+\delta}^{x+M}\frac{f_x(y)}{(y-x)^{m+1}}\diff{y}
        -\int_{x-M}^{x-\delta}\frac{f_x(y)}{(x-y)^{m+1}}\diff{y}\\
        &= (m+1)\int_{x+\delta}^{x+M}\frac{f(y)-f(x)}{(y-x)^{m+2}}\diff{y}
        +(m+1)\int_{x-M}^{x-\delta}\frac{f(y)-f(x)}{(x-y)^{m+2}}\diff{y}\\
        &~~+\left.\left[\frac{f(y)-f(x)}{(y-x)^{m+1}}\right]\right|_{y=x+\delta}^{x+M}
        -\left.\left[\frac{f(y)-f(x)}{(x-y)^{m+1}}\right]\right|_{y=x-M}^{x-\delta}.
    \end{align*}
    The proof completes by taking the limit and noticing that the boundary contribution vanishes as $M\to+\infty,\delta\to0^+$. In fact, we have
    \begin{align*}
        & \lim_{M\to+\infty,\delta\to0^+}\left\{\left.\left[\frac{f(y)-f(x)}{(y-x)^{m+1}}\right]\right|_{y=x+\delta}^{x+M}
        -\left.\left[\frac{f(y)-f(x)}{(x-y)^{m+1}}\right]\right|_{y=x-M}^{x-\delta}\right\}\\
        &= \lim_{M\to+\infty,\delta\to0^+}\left\{\frac{2f(x)-f(x-\delta)-f(x+\delta)}{\delta^{m+1}}
        +\frac{f(x+M)+f(x-M)-2f(x)}{M^{m+1}}
        \right\}\\
        &= \lim_{M\to+\infty,\delta\to0^+}\left\{O(\delta^{1-m})+O(M^{-m-1})\right\}=0.
    \end{align*}
\end{proof}
\begin{rmk}\label{rmk..fractional.Laplacian}
    For a function $f$ in the Schwartz space $\fS(\sR^d)$ of rapidly decreasing $C^\infty$ functions on $\sR^d$, we remark that the expression in \eqref{eq..relation.two.definition.non-local} equals (up to a constant $C_{d,\sigma}$) the fractional Laplacian \cite{Landkof1972Foundations}
    \begin{equation*}
        (-\Delta)^{\sigma}f(x)=C_{d,\sigma}\mathrm{P.V.}\int_{\sR^d}\frac{f(x)-f(y)}{\abs{x-y}^{d+2\sigma}}\diff{y}.
    \end{equation*}
    The constant
    $C_{d,\sigma}=(\int_{\sR^d}\frac{1-\cos(\xi_1)}{\abs{\xi}^{d+2\sigma}}\diff{\xi})^{-1}$, $\xi=(\xi_1,\dots,\xi_d)^{\T}\in \sR^d$ \cite{DiNezza2012Hitchhikers}.
    Our problem corresponds to $d=1$, $\sigma=\frac{1+m}{2}$. Note that for a function $f\in\fS(\sR)$, the equality in Proposition \ref{prop..relation.two.definition.non-local} still holds by the same proof. In particular, let $m=0$ and set $f=h$, then $C_{1,\frac{1}{2}}=\frac{1}{\pi}$ and
    \begin{equation*}
        (-\Delta)^{\frac{1}{2}}h(x)
        =\frac{1}{\pi}\mathrm{P.V.}\int_{-\infty}^{\infty}\frac{h(x)-h(y)}{\abs{x-y}^{2}}\diff{y}
        =\frac{1}{\pi}\mathrm{P.V.}\int_{-\infty}^{+\infty}\frac{h_x(y)}{x-y}\diff{y}
        =H(h_x)(x),
    \end{equation*}
    where we use Proposition \ref{prop..relation.two.definition.non-local} in the second equality and the definition of Hilbert transform in the last equality:
    \begin{equation}
        H(f)(x)=\frac{1}{\pi}\mathrm{P.V.}\int_{-\infty}^{+\infty}\frac{f(y)}{x-y}\diff{y}.
    \end{equation}
    Hence we recover the Hilbert transform formulation of Xiang's continuum model \cite{Xiang2002Derivation}.
\end{rmk}

\section{Existence, symmetry, and regularity of the energy minimizer}
From now on, we consider the minimizer of the total energy in a periodic setting. After showing its existence by the direct method in the calculus of variations,
we prove the minimizer is symmetric and unimodal on one period $\Omega=[-\frac{1}{2},\frac{1}{2}]$ by using a rearrangement inequality. This leads us to the Euler--Lagrange equation of the variational problem as well as the regularity of the minimizer. 

To obtain the existence of the energy minimizer, we start with reformulating the non-local energy in terms of step density $\rho=h_x$.

\begin{prop}[kernel-based representation]\label{prop..kernel-basedRepresentation}
    If $A>0$ and $-1<m<1$, then we have for $h\in X$ with $\rho=h_x$, the non-local energy can be written as
    \begin{equation}
      I_m[h]
      =\frac{1}{2}\int_{\Omega}\int_{\Omega} \rho(x)\rho(y)K_m(x-y)\diff{x}\diff{y},\label{eq..slope.formulation}
    \end{equation}
    Here the kernel $K_m(\cdot)$ is $L^1(\Omega)$, non-negative, $1$-periodic, even, and strictly decreasing on $(0,\frac{1}{2})$.
    Moreover, if $-1<m\leq 0$, then $K_m(\cdot)$ is $L^2(\Omega)$.
\end{prop}
\begin{proof}
	For $-1<m<1$, let $K_m(z)$ be the antiderivative of $-\sum_{k\in\sZ}\frac{z+k}{|z+k|^{m+2}}$ satisfying $K_m(\frac{1}{2})=0$. Clearly, $K_m$ is $1$-periodic. 
	Note that for $z\in(0,\frac{1}{2})$, $-K'_m(z)=\sum_{k\in\sZ}\frac{z+k}{|z+k|^{m+2}}=\sum_{k=0}^\infty\left[\frac{k+z}{|k+z|^{m+2}}-\frac{k+1-z}{|k+1-z|^{m+2}}\right]>0$ and $K_m\geq 0$ is strictly decreasing on $(0,\frac{1}{2})$.
	Thus $-K_m(z)+K_m(1-z)=\int_{z}^{1-z}\sum_{k=0}^\infty\Big[\frac{k+w}{|k+w|^{m+2}}-\frac{k+1-w}{|k+1-w|^{m+2}}\Big]\diff{w}=0$ due to the parity of the integrand. Then $K_m(-z)=K_m(1-z)=K_m(z)$ and $K_m$ is even. 
	In particular, if $m=0$, then we have the explicit expression $K_0(z)=-\log\sin(\pi z)$ with $K'_0(z)=-\pi\cot(\pi z)$, thanks to the identity $\pi\cot(\pi x)=\sum_{k\in\sZ}\frac{1}{x+k\pi}$.

	By definition, $K_m$ is piecewise continuous for any $-1<m<1$. In $[-\frac{1}{2}, \frac{1}{2}]$, its only possible singularity is at $z=0$. For $-1<m<1, m\neq 0$, $|K_m(z)|$ has the leading order of $|z|^{-m}$ as $z\to0$. For $m=0$, $|K_m(z)|$ has the leading order of $\log|z|$ as $z\to 0$. Hence $K_m\in L^1$ for any $-1<m<1$ and $K_m\in L^2$ for any $-1<m\leq 0$.
	
	Now an integration by parts completes the proof:
    \begin{align*}
        I_m[h]
        &= \frac{1}{2}\int_{\Omega}h(x)\int_{\Omega}\sum_{k\in\sZ}\frac{x-y+k}{\abs{x-y+k}^{m+2}}h_x(y)\diff{y}\diff{x}\\
        &= -\frac{1}{2}\int_{\Omega}h(x)\int_{\Omega}K'_m(x-y) h_x(y)\diff{y}\diff{x}\\
        &= \frac{1}{2}\int_{\Omega}\int_{\Omega}K_m(x-y) \rho(x)\rho(y)\diff{y}\diff{x}.
    \end{align*}
\end{proof}

Now we prove the existence of the energy minimizer in the solution space $X$ by the direct method in the calculus of variations.

\begin{proof}[Proof of Theorem \ref{thm..existence}]
    1. Let $E_*:=\inf_{h\in X}E[h]$. Note that $E_* \leq E[Ax]=\eps^{1-m}\Phi_{m,n}(A)<+\infty$. 

    2. Applying Young's convolution inequality to the kernel-based representation in Proposition \ref{prop..kernel-basedRepresentation}, we have $-I_m[h]\geq -\|K_m\|_{L^1}\|\rho\|_{L^2}^2$.
    There exists a constant $C=C(m,n)$ such that for any $\xi\in\sR$, we have
    \begin{equation*}
        \Phi_{m,n}(\xi)\geq \frac{\gamma}{2n(n+1)}\xi^{n+1}-C.
    \end{equation*}
    For any $h\in X$, we obtain a lower bound
    \begin{equation*}
        E[h]\geq -\|K_m\|_{L^1}\|\rho\|_{L^2}^2
        +\frac{\eps^{1-m}\gamma}{2n(n+1)}\norm{\rho}^{n+1}_{L^{n+1}}-C\eps^{1-m}
        \geq C\norm{\rho}^{2}_{L^2}-C,
    \end{equation*}
    where the constants $C=C(m,n,\eps,A)$ in the last line.
    This gives us the coercivity of $E$. Also, we have $E_*\geq -C$ is finite. 

    3. Now we show $E$ is lower semicontinuous in $h$. Note that $\Phi''(\xi)=\xi^{m-1}+\gamma \xi^{n-1}>0$ for $\xi>0$. Thanks to the convexity of $\Phi(\cdot)$, the local term energy $\eps^{1-m}\int_{\Omega}\Phi(h_x)\diff{x}$ is lower semicontinuous. It remains to show the continuity of the non-local term. We use the kernel-based representation in Proposition \ref{prop..kernel-basedRepresentation} to show this. If $\{\tilde{h}^k\}_{k=1}^\infty\subset H^1_\#(\Omega)$ weakly convergence to $\tilde{h}_*\in H^1_\#(\Omega)$, then $\rho^k$ converges to $\rho_*$ weakly in $L^2(\Omega)$. Since $\rho^k,\rho_*\in L^2$ and $K_m\in L^1$, Young's convolution inequality leads to $\rho^k\ast K_m\in L^2,\,\,\rho_*\ast K_m\in L^2$.
    Thus
    \begin{align*}
        \lim_{k\to+\infty}\left\{I_m[h^k]-I_m[h_*]\right\}
        &= \lim_{k\to+\infty}\int_{\Omega}\left(\rho^k(x)-\rho_*(x)\right)\left[\rho^k\ast K_m\right](x)\diff{x}\\
        &~~~~+\lim_{k\to+\infty}\int_{\Omega}\rho_*(x)\left[(\rho^k-\rho_*)\ast K_m \right](x)\diff{x}
        = 0.
    \end{align*}
    Therefore, we have the lower semicontinuity: $\lim\inf_{k\to+\infty}E[h^k]\geq E[h_*]$.

    4. Take any minimizing sequence $\{h^k\}_{k=1}^\infty$. Then $\lim_{k\to+\infty}E[h^k]=E_*$. By coercivity, we conclude that $\sup_{k}\norm{\rho^k}_{L^2}<\infty$. Since we are in the periodic setting, we have $\sup_k\norm{\tilde{h}}_{L^2}<\infty$. Hence $\{\tilde{h}^k\}_{k=1}^{\infty}$ is bounded in $H^1_\#(\Omega)$.

    Consequently, there exist a subsequence, still denoted as $\{h^k\}_{k=1}^\infty$, and a function $h_*$ such that $\tilde{h}_*\in H^1$ and $\tilde{h}^k$ converges to $\tilde{h}_*$ weakly in $H^1_\#(\Omega)$.  Note that $X$ is a convex, closed subset of $V+\{Ax\}$. Then $X$ is weakly closed due to Mazur's Theorem. Thus $h_*\in X$.
    By the lower semicontinuity, $E[h_*]\leq \lim\inf_{k\to+\infty}E[h^k]=E_*$. Therefore, $h_*\in X$ is an energy minimizer.
\end{proof}

\begin{proof}[Proof of Theorem \ref{thm..sysmmetry}]
    It is clear that the rearrangement of the derivative keeps the local term:
    $
        \eps^{1-m}\int_{\Omega} \Phi_{m,n}(\rho)\diff{x}=\eps^{1-m}\int_{\Omega} \Phi_{m,n}(\rho^*)\diff{x}.
    $
    The Riesz's rearrangement inequality in Appendix \ref{sec:appendix.Riesz.rearrangement} with the kernel-based representation in Proposition \ref{prop..kernel-basedRepresentation} leads to
    \begin{equation}
        I_m[h]\leq I_m[h^*_\mathrm{der}]
    \end{equation}
    We further $h^*_\mathrm{der}\in X$ is the antiderivative of $\rho^*$ satisfying $h^*_\mathrm{der}(0)=h(0)$.
    Therefore $E[h]\geq E[h^*_\mathrm{der}]$. The equality holds when $h=h^*_\mathrm{der}$ up to a translation. 
    As a result, the support of $\rho$ in the period $\Omega$ is a closed interval centered at zero, i.e., $\mathrm{supp}(\rho):=\{x\in\Omega\mid\overline{\{\rho(x)\neq0\}}=[-R_0,R_0]$ for some $R_0\in(0,\frac{1}{2}]$. 
\end{proof}

\begin{prop}[Euler--Lagrange equation]\label{prop..first.variation}
    If $A>0$, $\eps>0$, and $-1<m<1<n$, then for the global energy minimizer $h\in X$ of $E$ with $\mathrm{supp}(h_x)=[-R_0,R_0]$ we have
    \begin{equation}\label{eq..Lagrange.multiplier}
        -K_{m}\ast\rho (x)+\eps^{1-m}\Phi'_{m,n}(\rho(x))+\lambda=0,\quad x\in [-R_0,R_0],
    \end{equation}
    where $\lambda$ is a Lagrange multiplier.
    The Euler--Lagrange equation holds in the weak sense
    \begin{equation}\label{eq..elliptic.PDE}
        (K_{m}\ast h_x)_x-\eps^{1-m}\left(\frac{1}{h_x^{1-m}}+\gamma h_x^{n-1}\right)h_{xx}=0,\quad x\in [-R_0,R_0],
    \end{equation}
    Moreover, if $-1<m\leq 0$, then there further exist constants $\rho_{\min}$ and $\rho_{\max}$ such that $0<\rho_{\min}\leq \rho(x)\leq \rho_{\max}<+\infty$ for any $x\in [-R_0,R_0]$.
\end{prop}
\begin{proof}
	Eq. \eqref{eq..Lagrange.multiplier} is obtained by direct calculation of $\frac{\delta E}{\delta \rho}$ in the standard way. The appearance of a Lagrange multiplier $\lambda$, which is due to the constraint $\int_{\Omega}\rho\diff{x}=A$. Note that $(\Phi'_{m,n}(\rho))_x=(\Phi''_{m,n}(h_x))h_{xx}=\left(\frac{1}{h_x^{1-m}}+\gamma h_x^{n-1}\right)h_{xx}$. By using test functions vanishing on $\Omega\backslash[-R_0,R_0]$, one can obtain the Euler--Lagrange equation \eqref{eq..elliptic.PDE} by direct calculation of $\frac{\delta E}{\delta h_x}$. For $-1<m\leq 0$, Young's convolution inequality with $\rho,K_m\in L^2$ leads to $\norm{K_{m}\ast\rho}_{L^\infty}\leq \norm{K_{m}}_{L^2}\norm{\rho}_{L^2}<+\infty$. Hence $\Phi'_{m,n}(\rho(x))$ is bounded on $[-R_0,R_0]$. This with \eqref{eq..derivativePhi} implies $\rho(x)$, $x\in[-R_0,R_0]$ is bounded by some positive constants $\rho_{\min}$ and $\rho_{\max}$ from below and above.
\end{proof}

This proposition allows us to prove the regularity of the minimizer for $-1<m\leq 0$.

\begin{proof}[Proof of Theorem \ref{thm..regularity}]
    By Proposition \ref{prop..first.variation}, we have the weak solution of \eqref{eq..elliptic.PDE}. Since $0<\rho_{\min}\leq \rho(x)\leq \rho_{\max}<+\infty$ for any $x\in [-R_0,R_0]$, we have $C\leq \frac{1}{h_x^{1-m}}+\gamma h_x^{n-1}\leq C'$ for some positive constants $C$ and $C'$. Thus \eqref{eq..elliptic.PDE} is an elliptic partial differential equation with uniform ellipticity condition. Therefore the standard elliptic regularity theory implies that $h$ is smooth.
\end{proof}

\section{Minimum energy scaling law of the energy minimizer}
In this section, we characterize the energetic induced step bunching of the generalized continuum model by showing its non-trivial minimum energy scaling law. In particular, we achieve matching upper and lower bounds for the minimum energy. The upper bound is obtained by a good ansatz, while the lower bound is due to a new interpolation inequality which is based on a specific formulation of the non-local energy.

The total energy has the following representation. In particular, the non-local term $\tilde{I}_m$ equals the previous non-local energy $I_m$ up to a null Lagrangian.
\begin{prop}[Another formulation of the total energy]\label{prop..nullLagrangian}
	If $A>0$, $\eps>0$, and $-1<m<1<n$, then for any $h\in X$ we have another representation of the total energy,
	\begin{align}
	    E[h]
	    &= -\tilde{I}_m[h]-W[h]+\eps^{1-m}\int_{\Omega} \Phi_{m,n}(h_x)\diff{x},\\
	    \tilde{I}_m[h]
	    &:= \frac{1}{2}\int_{\Omega}\tilde{h}(x)\mathrm{P.V.}\int_{\sR}\frac{(x-y)\tilde{h}_x(y)}{\abs{x-y}^{m+2}}\diff{y}\diff{x},\\
		W[h]
	    &:= \frac{A}{2}\norm{\rho\ast K_m}_{L^1}.
	\end{align}
	Here $W[h]$ is a null Lagrangian, i.e., $\frac{\delta W}{\delta h}=0$ for all $h\in X$.
\end{prop}
\begin{proof}
	Similar to the same proof of Proposition \ref{prop..kernel-basedRepresentation}, we have
	\begin{equation}
      \tilde{I}_m[h]
      =\frac{1}{2}\int_{\Omega}\int_{\Omega} \tilde{h}_x(x)\rho(y)K_m(x-y)\diff{x}\diff{y}.
    \end{equation}
    Note that $\rho(x)-\tilde{h}_x(x)=A$. Thus
	\begin{align*}
		W[h]=I_m[h]-\tilde{I}_m[h]
		&=\frac{1}{2}\int_{\Omega}\int_{\Omega} (\rho(x)-\tilde{h}_x(x))\rho(y)K_m(x-y)\diff{x}\diff{y}\\
		&=\frac{A}{2}\int_{\Omega}\rho\ast K_m(x)\diff{x}=\frac{A}{2}\norm{\rho\ast K_m}_{L^1}.
	\end{align*}
	In the last step, we use the fact that $K_m$ and $\rho$ are both non-negative on $\Omega$. 
	Let $f(\cdot)$ be a smooth perturbation. Then we have
	\begin{align*}
		\left\langle\frac{\delta W}{\delta h},f\right\rangle
		&=\frac{A}{2}\int_{\Omega}\int_{\Omega}K_m(x-y)f_x(y)\diff{y}\diff{x}\\
		&=\frac{A}{2}\int_{\Omega}\left[\int_{\Omega}K'_m(x-y)\diff{x}\right]f(y)\diff{y}=0.
	\end{align*}
	Hence $W[h]$ is a null Lagrangian.
\end{proof}

We provide an equivalent formulation for the non-local energy, up to the null Lagrangian.

\begin{prop}[Fourier series-based representation]\label{prop..EquivalentFormulationNonlocalEnergy}
    If $A>0$, $-1<m<1$, and $h\in X$ is piecewise $C^2$, then we have
    \begin{equation*}
      \tilde{I}_m[h]
      =S_m\sum\limits_{k\in\sZ}|k|^{m+1}\abs{h_{k}}^2,
    \end{equation*}
    where $h_k:=\int_{\Omega}\tilde{h}(x)\E^{-2\pi\I kx}\diff{x}$ is the Fourier coefficients of $\tilde{h}$, and the constant $S_m:=(2\pi)^{m+1}\int_{0}^{+\infty}z^{-m-1}\sin z\diff{z}$.
\end{prop}
\begin{proof}
    We define $d(x)=\displaystyle\sum\limits_{k\in\sZ}\frac{x+k}{\abs{x+k}^{m+2}}$ and consider the Fourier series
    \begin{equation*}
      \tilde{h}(x)
      = \sum_{k\in\sZ}h_{k}\E^{2\pi \I k x},\quad
      d(x)
      = \sum_{k\in\sZ}d_{k}\E^{2\pi \I k x}.
    \end{equation*}
    Clearly, we have $h_k=\overline{h_{-k}}$ and
    $
      \tilde{h}_x(x)=\sum_{k\in\sZ} 2\pi \I k h_{k}\E^{2\pi \I k x}
    $. By Proposition \ref{prop..relation.two.definition.non-local} and properties of Fourier series, we have
    \begin{align*}
      \tilde{I}_m[h]
      &= \frac{1}{2}\int_{\Omega}\tilde{h}(x)\int_{\Omega} d(y)\tilde{h}_x(x-y)\diff{y}\diff{x}\\
      &= \frac{1}{2}\int_{\Omega}\tilde{h}(x)\int_{\Omega}\sum\limits_{k\in\sZ}d_{k}\E^{2\pi \I k y}\sum\limits_{k'\in\sZ}2\pi \I k'h_{k'}\E^{2\pi \I k'(x-y)}\diff{y}\diff{x}\\
      &= \frac{1}{2}\int_{\Omega} \tilde{h}(x)\sum\limits_{k\in\sZ}2\pi \I k d_{k} h_{k}\E^{2\pi \I k x}\diff{x}\\
      &= \frac{1}{2}\int_{\Omega}\sum\limits_{k'\in\sZ}h_{-k'}\E^{-2\pi \I k' x}\sum\limits_{k\in\sZ}2\pi \I k d_{k} h_{k}\E^{2\pi \I k x}\diff{x}\\
      &= \sum\limits_{k\in\sZ}\pi \I k h_{-k}d_{k} h_{k}
      = \sum\limits_{k\in\sZ}\pi \I k d_{k}\abs{h_{k}}^2.
    \end{align*}
    Note that for $k\in\sZ$
    \begin{equation*}
      d_{k}
      = \int_{\Omega}\sum\limits_{k'\in\sZ}\frac{x+k'}{\abs{x+k'}^{m+2}}\E^{-2\pi \I k x}\diff{x}
      = \int_{\sR}\frac{x}{\abs{x}^{s+2}}\E^{-2\pi \I k x}\diff{x}.
    \end{equation*}
    Thus for $k>0$, we have
    \begin{align*}
        \pi \I k d_{k}+\pi \I (-k) d_{-k}
        &= \int_{\sR}\frac{\pi \I k x}{\abs{x}^{m+2}}\E^{-2\pi \I k x}+\frac{-\pi \I k x}{\abs{x}^{m+2}}\E^{2\pi \I k x}\diff{x}\\
        &= \int_{\sR}\frac{2\pi k x\sin(2\pi k x)}{\abs{x}^{m+2}}\diff{x}\\
        &= (2\pi k)^{m+1}\int_{\sR}\frac{2\pi k x\sin(2\pi k x)}{\abs{2\pi k x}^{m+2}}\diff{(2\pi k x)}\\
        &= 2(2\pi k)^{m+1}\int_{0}^{+\infty}z^{-m-1}\sin z\diff{z}\\
        &= 2 k^{m+1} S_m.
    \end{align*}
    Therefore, we have
    $\tilde{I}_m[h]
      = \sum_{k\in\sZ}\pi \I k d_{k}\abs{h_{k}}^2
      = S_m\sum_{k\in\sZ}|k|^{m+1}\abs{h_{k}}^2
    $.
\end{proof}

Recall Remark \ref{rmk..fractional.Laplacian} that the non-local intergal is equivalent to fractional Laplacian for sufficently smooth functions. With this perspective, Proposition \ref{prop..EquivalentFormulationNonlocalEnergy} tells us that the non-local energy is equivalent to the square of $H^{\frac{1+m}{2}}$ semi-norm of $\tilde{h}$ with $-1<m<1$.
To obtain an optimal estimate for the lower bound of the minimum energy in the case of $m=0$, we prove the following interpolation inequality which can be regarded as a quantitative extension of the Sobolev imbedding inequality for $H^{\frac{1}{2}}$ functions imbedding into $L^2$ space.
\begin{prop}[interpolation inequality]\label{prop..I_m.upper.bound}
    If $A>0$, $\eps>0$, $m=0$, and $1<n$, then for the energy minimizer $h\in X$ of $E$ with $\rho=h_x$, then we have
    \begin{align*} 
        \tilde{I}_0[h]
        &\leq \frac{2}{N}\norm{h_x}_{L^2}^2+3A^2+A^2\log N,\\
		\tilde{I}_0[h]
        &\leq \frac{2}{N}\norm{h_x}_{L^2}^2+4A^2+A^2\log N.
    \end{align*}
\end{prop}
\begin{proof}
	Thanks to Theorem \ref{thm..regularity}, $h$ is piecewise $C^2$ in the case for $m=0$. Hence we apply Proposition \ref{prop..EquivalentFormulationNonlocalEnergy} with $S_0=2\pi\int_0^{+\infty}\frac{\sin z}{z}\diff{z}=\pi^2$ and split the sum into two parts:
    \begin{align*}
        \tilde{I}_0[h]
        &=\pi^2\sum_{k\leq N}\abs{k}\abs{h_k}^2+\pi^2\sum_{k> N}\abs{k}\abs{h_k}^2\\
        &\leq \frac{1}{4}(\max_{k\in\sZ}{\abs{2\pi k h_k}})^2\sum_{k\leq N}\abs{k}^{-1}+\frac{1}{4}\sum_{k> N}\abs{k}^{-1}\abs{2\pi k h_k}^2.
    \end{align*}
    For the second term, we have 
    \begin{equation}
    	\sum_{k> N}\abs{k}^{-1}\abs{2\pi k h_k}^2
    	\leq \frac{1}{N}\norm{\tilde{h}_x}^2_{L^2}
		\leq\frac{2}{N}(\norm{h_x}_{L^2}^2+A^2)
    	\leq \frac{2}{N}\norm{h_x}_{L^2}^2+2A^2.\label{eq..interpolationPart1}
	\end{equation}
   
    By the property of the Fourier series, we have
    \begin{align*}
    	\max_{k\in\sZ}{\abs{2\pi k h_k}}
    	=\max_{k\in\sZ}\Abs{\int_{\Omega}\tilde{h}_x(x)\E^{-2\pi\I kx}\diff{x}}
    	\leq \norm{\tilde{h}_x}_{L^1}=\norm{h_x}_{L^1}+\norm{A}_{L^1}\leq 2A.
    \end{align*}
    Hence 
    \begin{equation}
	    \frac{1}{4}(\max_{k\in\sZ}{\abs{2\pi k h_k}})^2\sum_{k\leq N}\abs{k}^{-1}
	    \leq\frac{1}{4}(2A)^2(1+\log N)=A^2\log N+A^2.\label{eq..interpolationPart2}
    \end{equation}
    Then $\tilde{I}_0[h]$ is bounded by combining \eqref{eq..interpolationPart1} and \eqref{eq..interpolationPart2}.
    Recall that $I_0[h]=\tilde{I}_0[h]+W[h]$ where $W[h]=\frac{A}{2}\norm{\rho\ast K_0}_{L^1}$ and $K_0=-\log\sin(\pi z)$ for $z\in[-\frac{1}{2},\frac{1}{2}]$. By Young's convolution inequality, we have
    $W[h]\leq\frac{A}{2}\norm{\rho}_{L^1}\norm{K_0}_{L^1}\leq \frac{A^2}{2}$. This with the bound of $\tilde{I}_0[h]$ completes the proof.
\end{proof}

Now, we are ready to obtain the minimum energy scaling law for the continuum model on epitaxial growth with elasticity effects.

\begin{proof}[Proof of Theorem \ref{thm..energy.scaling.law}]
    1. We first obtain an upper bound for the minimum energy. This is obtained by considering the energy of the following height profile in one period $\Omega$
    \begin{equation*}
        h^0(x)=\left\{
        \begin{array}{ll}
            \rho_0 x, & \abs{x}\leq \frac{A}{2\rho_0},\\
            A/2, & x>\frac{A}{2\rho_0},\\
            -A/2, & x<-\frac{A}{2\rho_0}.
        \end{array}
        \right.
    \end{equation*}
    For sufficiently small $\eps$, we set $\rho_0:=\eps^{-\frac{1}{n}}\geq A$.
    Thus
    \begin{align*}
        \eps\int_{\Omega}\Phi_{0,n}(h_x)\diff{x}
        &= \eps A\rho_0^{-1}\Phi_{0,n}(\rho_0)\\
        &= \frac{\gamma A}{n(n+1)}\eps\rho_0^n+A\eps \log\rho_0\\
        &= \frac{\gamma A}{n(n+1)}-\frac{A\eps}{n} \log\eps.
    \end{align*}
    Utilizing Proposition \ref{prop..kernel-basedRepresentation}, we have for $m=0$
    \begin{align*}
        -I_0[h^0]
        &= 
        \rho_0^2\int_{-\frac{A}{2\rho_0}}^{\frac{A}{2\rho_0}}\int_{-\frac{A}{2\rho_0}}^{\frac{A}{2\rho_0}}\log\sin(\pi (x-y))\diff{x}\diff{y}\\
        &\leq 
        A^2\log\frac{A}{\rho_0}+C
        \leq 
        \frac{A^2}{n}\log\eps+C.
    \end{align*}
    Substituting these into $E[h^0]$ leads to an upper bound of the minimum energy:
    \begin{equation}
        \inf_{h\in X}E[h]\leq E[h^0]\leq -\frac{A^2}{n}\abs{\log \eps}+C.
    \end{equation}

    2. The lower bound of the energy is obtained by taking an optimal interpolation parameter $N$ in Proposition \ref{prop..I_m.upper.bound}. There exists a constant $C$ such that for any $\xi\in\sR$, we have
    \begin{equation*}
        \Phi_{0,n}(\xi)\geq \frac{\gamma}{2n(n+1)}\xi^{n+1}-C.
    \end{equation*}
    For $m=0$, setting $N=\eps^{-\frac{1}{n}}$ and applying Proposition \ref{prop..I_m.upper.bound} for any $h\in X$, we obtain the lower bound
    \begin{align*}
        E[h]
        &\geq -\frac{2}{N}\norm{h_x}_{L^2}^2-4A^2-A^2\log N+\frac{\eps\gamma}{2n(n+1)}\norm{\rho}^{n+1}_{L^{n+1}}-C\eps\\
        &\geq -2\eps^{\frac{1}{n}}\norm{\rho}_{L^2}^2-4A^2-A^2\log \eps^{-\frac{1}{n}}+\frac{\eps\gamma}{2n(n+1)}\norm{\rho}^{n+1}_{L^{n+1}}-C\eps.
    \end{align*}
    Note that for $\xi\geq 0$ the function $f(\xi):=\frac{\eps\gamma}{2n(n+1)}\xi^n-2\eps^{\frac{1}{n}}\xi$ takes minimum at $\xi_*:=(4(n+1))^{\frac{1}{n-1}}\eps^{-\frac{1}{n}}\gamma^{-\frac{1}{n-1}}$. Thus $f(\xi)\geq f(\xi_*)=C_*$ which is independent of $\eps$. Therefore
    \begin{align*}
        E[h]
        &\geq \int_{\Omega}\left[\frac{\eps\gamma}{2n(n+1)}\rho^{n+1}(x)-2\eps^{\frac{1}{n}}\rho^2(x)\right]\diff{x}-4A^2-C\eps-A^2\log \eps^{-\frac{1}{n}}\\
        &\geq \int_{\Omega}C_*\rho(x)\diff{x}-4A^2-C\eps-\frac{A^2}{n}\abs{\log \eps}\\
        &\geq -\frac{A^2}{n}\abs{\log \eps}-C.
    \end{align*}
\end{proof}

\appendix

\section{Riesz rearrangement inequality}\label{sec:appendix.Riesz.rearrangement}
We introduce the Riesz rearrangement inequality whose proof can essentially be found in \cite{Kawohl1985Rearrangements}.
\begin{thm}[Riesz rearrangement inequality on a periodic domain \cite{Kawohl1985Rearrangements}]
  Let $K$ be a function on the periodic domain $\Omega=[-\frac{1}{2},\frac{1}{2}]$. If $K$ is even and strictly decreasing on $(0,\frac{1}{2})$, then for any pair $f,g$ of nonnegative measurable functions, we have
  \begin{equation}
    \int_{\Omega}\int_{\Omega}f(x)g(y)K(x-y)\diff{x}\diff{y}\leq \int_{\Omega}\int_{\Omega}f^*(x)g^*(y)K(x-y)\diff{x}\diff{y}.
  \end{equation}
  where the equality (with a finite and nonzero value of the integral) occurs only if there exists a translation $T$ such that $f=f^*\circ T$ and $g=g^*\circ T$ almost everywhere.
\end{thm}

\section{Euler--Maclaurin expansion for singular integrals}\label{sec..EulerMaclaurin}
We rephrase the following theorem from \cite{Sidi1988Quadrature}. As mentioned in \cite{Sidi1988Quadrature}, this theorem is true for any $m<1$, although they were originally proved for $0<m<1$ by Navot \cite{Navot1961extension}.
\begin{thm}[Essentially rephrased from Theorem 2 of \cite{Sidi1988Quadrature}]\label{thm..SigularEulerMaclaurin}
    Let $g(h)$ be a function in $C^{2p}(\overline{\sR^+})\cap W^{2p,1}(\overline{\sR^+})$ with $p\geq 1$. Let $a>0$, $h_j=ja$, $j=0,1,\cdots$, and $G(h)=h^{-m}g(h)$, $m<1$. Then as $a\rightarrow 0$, we have the asymptotic expansion
    \begin{align*}
        &\int_{0}^{+\infty}G(h)\diff{h}
        =a \sum_{j=1}^{\infty}G(h_j) 
        -\sum_{r=0}^{2p-1}\frac{\zeta(m-r)}{r!}g^{(r)}(0)a^{r-m+1}+R_{2p},\\
        &R_{2p}
        = a^{2p}\int_{0}^{+\infty}\frac{\bar{B}_{2p}(h/a)-B_{2p}}{(2p)!}\frac{\D^{2p}}{\D h^{2p}}
        \left\{\int_{0}^{h}\frac{g^{(2p)}(h')}{(2p-1)!}h^{-m}(h-h')^{2p-1}\diff{h'}
        \right\}\diff{h},
    \end{align*}
    where $B_{2p}$ is the Bernoulli number and $\bar{B}_{2p}(\xi)$ is the periodic Bernoullian function.
\end{thm}

\section*{Acknowledgments}
This work was supported by the Hong Kong Research Grants Council General Research Fund 16313316.

\bibliographystyle{spmpsci}

\end{document}